\documentclass[12pt]{amsart}

\usepackage{hyperref}

\newtheorem{theorem}{Theorem}[section]
\newtheorem{corollary}[theorem]{Corollary}
\newtheorem{lemma}[theorem]{Lemma}
\newtheorem{proposition}[theorem]{Proposition}

\theoremstyle{theorem}
\newtheorem{innercustomthm}{Theorem}
\newenvironment{customthm}[1]
  {\renewcommand\theinnercustomthm{#1}\innercustomthm}
  {\endinnercustomthm}

\theoremstyle{definition}
\newtheorem{definition}[theorem]{Definition}

\theoremstyle{remark}

\numberwithin{equation}{section}

\newcommand{\p}{\partial}
\newcommand{\n}{\nabla}
\newcommand{\der}[1]{\frac{\partial}{\partial #1}}

\newcommand{\la}{\langle}
\newcommand{\ra}{\rangle}

\newcommand{\bm}[0]{\bar{\mu}}
\newcommand{\bl}[0]{\bar{\lambda}}
\newcommand{\bs}[0]{\bar{\sigma}}
\newcommand{\sech}[0]{\text{sech}}
\newcommand{\csch}[0]{\text{csch}}

\begin{document}
%\date{\today}
\title[$S^1$-invariant free boundary minimal annuli and M\"obius bands in $\mathbb{B}^n$]
{Existence and classification of $S^1$-invariant free boundary annuli and M\"obius bands in $\mathbb{B}^n$}
\author{Ailana Fraser}
\address{Department of Mathematics \\
                 University of British Columbia \\
                 Vancouver, BC V6T 1Z2}
\email{afraser@math.ubc.ca}
\author{Pam Sargent}
\address{Department of Mathematics \\
                 Yale University \\
                 New Haven, CT 06511}
\email{pamela.sargent@yale.edu}
\thanks{2010 {\em Mathematics Subject Classification.} 53A10, 35P15. \\
The authors were partially supported by  the 
Natural Sciences and Engineering Research Council of Canada. Part of this work was done while A. Fraser was visiting the Institute for Advanced Study for the special year {\em Variational Methods in Geometry}, with funding from NSF grant DMS-1638352 and the James D. Wolfensohn Fund, and A. Fraser gratefully acknowledges the support of the IAS}

\begin{abstract} 
We explicitly classify all $S^1$-invariant free boundary minimal annuli and M\"obius bands in $\mathbb{B}^n$. This classification is obtained from an analysis of the spectrum of the Dirichlet-to-Neumann map for $S^1$-invariant metrics on the annulus and M\"obius band. 
First, we determine the supremum of the $k$-th normalized Steklov eigenvalue among all $S^1$-invariant metrics on the M\"obius band for each $k \geq 1$, and show that it is achieved by the induced metric from a free boundary minimal embedding of the M\"obius band into $\mathbb{B}^4$ by $k$-th Steklov eigenfunctions. We then show that the critical metrics of the normalized Steklov eigenvalues on the space of $S^1$-invariant metrics on the annulus and M\"obius band are the induced metrics on explicit free boundary minimal annuli and M\"obius bands in $\mathbb{B}^3$ and $\mathbb{B}^4$, including some new families of free boundary minimal annuli and M\"obius bands in $\mathbb{B}^4$. Finally, we prove that these are the only $S^1$-invariant free boundary minimal annuli and M\"obius bands in $\mathbb{B}^n$.
\end{abstract}

\maketitle

\section{Introduction} \label{section:introduction}

In this paper we explicitly classify all $S^1$-invariant free boundary minimal annuli and M\"obius bands in the unit ball $\mathbb{B}^n$ in Euclidean space. This classification is obtained from an analysis of the  Steklov spectrum for $S^1$-invariant metrics on the annulus and M\"obius band.  The Steklov problem on surfaces with boundary is a natural and much studied eigenvalue problem that has  interesting connections to the free boundary problem for minimal surfaces in $\mathbb{B}^n$. A {\em free boundary minimal surface} in $\mathbb{B}^n$ is a minimal surface in $\mathbb{B}^n$ with boundary contained in the boundary of the ball and meeting the boundary of the ball orthogonally. Such surfaces arise variationally as critical points of the area among surfaces in the ball 
whose boundaries lie on $\partial \mathbb{B}^n$ but are free to vary on $\partial \mathbb{B}^n$. The simplest example is an equatorial plane disk. Another explicit example is the {\em critical catenoid}, the unique piece of a suitably scaled catenoid that defines a free boundary surface in $\mathbb{B}^3$. A. Fraser and R. Schoen \cite{FS3} established a connection between free boundary minimal surfaces in $\mathbb{B}^n$ and the Steklov eigenvalue problem, and proved existence of an embedded free boundary minimal surface in $\mathbb{B}^3$ of genus zero with any number of boundary components. Since then, further new examples of free boundary minimal surfaces in the ball have been constructed using gluing methods, equivariant min-max theory, and eigenvalue optimization (\cite{FTY}, \cite{FPZ}, \cite{Ke}, \cite{KL}, \cite{KW}). The general question of whether there exists a free boundary minimal surface in $\mathbb{B}^3$ of any given topological type, or how many distinct free boundary minimal surfaces in $\mathbb{B}^n$ of a given topological type there are, remains open.

Though we have a number of existence results for free boundary minimal surfaces in $\mathbb{B}^n$, explicit constructions are less common. In this paper we show that the critical metrics of the Steklov eigenvalues on the space of $S^1$-invariant metrics  on the annulus and M\"obius band with normalized boundary length are the induced metrics on explicit families of free boundary minimal annuli and M\"obius bands in $\mathbb{B}^3$ and $\mathbb{B}^4$, including some new families of free boundary minimal annuli and M\"obius bands in $\mathbb{B}^4$. Moreover, we prove that these are the only $S^1$-invariant free boundary minimal annuli and M\"obius bands in $\mathbb{B}^n$. We say that a free boundary minimal annulus or M\"obius band is {\em $S^1$-invariant} if the induced metric is $S^1$-invariant. Our proof involves showing that the induced metric on any $S^1$-invariant free boundary annulus or M\"obius band in $\mathbb{B}^n$ is critical for some normalized Steklov eigenvalue on the space of $S^1$-invariant metrics. The explicit classification of all $S^1$-invariant free boundary minimal annuli and M\"obius bands in $\mathbb{B}^n$ then follows.

If $(M, g)$ is a compact Riemannian surface with boundary, a function $u$ on $M$ is a Steklov eigenfunction with eigenvalue $\sigma$ if:
\begin{equation*}
       \begin{cases}
       \Delta_g u=0 & \text{on $M$} \\
       \frac{\partial u}{\partial\eta}=\sigma u & \text{on $\partial M$},
       \end{cases}
\end{equation*}
where $\eta$ is the outward unit normal vector to $\partial M$. Steklov eigenvalues are eigenvalues of the Dirichlet-to-Neumann operator, which maps a given  function on the boundary to the normal derivative of its harmonic extension to the interior. The Steklov spectrum is discrete and the eigenvalues form a sequence
\begin{equation*}
           0=\sigma_0 < \sigma_1 \leq \sigma_2 \leq \ldots \leq \sigma_k \leq \ldots \rightarrow\infty.
\end{equation*}

In 1954, R. Weinstock \cite{W} proved that if $D$ is a simply connected surface then for any metric $g$ on $D$, $\sigma_1(D,g) L_g(\partial D) \leq 2\pi$, with equality if and only if $(D,g)$ is $\sigma$-homothetic (see Definition \ref{definition:equivalent}) to a Euclidean round disk. In $1975$ J. Hersch, L. Payne, and M. Schiffer \cite{HPS} generalized Weinstock's theorem to higher eigenvalues showing that $\sigma_k(D, g)L_g(\partial D) \leq 2\pi k$, for any $k \geq 1$. In $2008$ A. Girouard and I. Polterovich \cite{GP} proved that this upper bound is sharp for each $k$, and is attained in the limit by a sequence of domains degenerating to a union of $k$ disjoint identical round disks, but for $k= 2$ the upper bound is not achieved. Recently A. Fraser and R. Schoen \cite{FS4} proved that the upper bound is not achieved for any $k \geq 2$. We note that there is a very similar picture in the case of the Laplace eigenvalues of a simply connected closed surface. In 1970, J. Hersch \cite{H} proved that for any smooth metric $g$ on the sphere we have $\lambda_1(S^2,g) A_g(S^2) \leq 8\pi$, with equality if and only if $g$ is a constant curvature metric. M. Karpukhin, N. Nadirashvili, A. Penskoi and I. Polterovich \cite{KNPP} recently proved a sharp upper bound $\lambda_k(S^2,g) A_g(S^2)\leq 8\pi k$ for {\em every} $k \geq 1$, and showed that for $k \geq 2$ the upper bound is not achieved (see also \cite{Ka}, \cite{N}, \cite{P}, \cite{NS}).

In general, as far as we are aware, there are no surfaces for which maximizing metrics for higher eigenvalues are known to exist. In Section \ref{section:mobius} we consider the maximization problem on surfaces with boundary when one imposes certain symmetries. Specifically we consider $S^1$-invariant metrics on the M\"obius band. The case of rotationally symmetric metrics on the annulus has been studied by A. Fraser and R. Schoen \cite{FS1} and X.-Q. Fan, L.~F. Tam and G.~Yu \cite{FTY}. In \cite{FTY} they determined the supremum of the $k$-th nonzero normalized Steklov eigenvalue over all rotationally symmetric metrics on the annulus and showed that the supremum is achieved for all $k>2$, but that the second nonzero normalized Steklov eigenvalue is not achieved. In Section \ref{section:mobius} we show that the supremum $\sigma_k^{S^1}$ of the $k$-th normalized Steklov eigenvalue among $S^1$-invariant metrics on the M\"obius band is achieved for {\em all} $k \geq 1$. This extends \cite[Proposition 7.1]{FS3} to $k >1$.

\begin{theorem} \label{theorem:mobius}
For all $k \geq 1$, 
\[
      \sigma^{S^1}_{2k-1}=\sigma^{S^1}_{2k}=4\pi k \tanh(2kT_{k,1})
\]
and the supremum is uniquely (up to $\sigma$-homothety) achieved by the induced metric on the free boundary minimal M\"obius band in $\mathbb{B}^4$ given by the embedding 
$u: [-T_{k,1},T_{k,1}] \times S^1 / \sim \; \rightarrow \; \mathbb{B}^4$ with
\[
     u(t,\theta)
     =\frac{1}{R_k}
        \left( 2k \sinh(t) \cos(\theta), 2k \sinh(t) \sin(\theta), \cosh(2kt) \cos(2k \theta), \cosh(2kt) \sin(2k \theta) \right)
\]
where $T_{k,1}$ is the unique positive solution of $2k \tanh(2kt)=\coth(t)$ and
\[
      R_k=\sqrt{4k^2\sinh^2(T_{k,1})+\cosh^2(2kT_{k,1})}.
\]
\end{theorem}

Thus, although there are no surfaces for which maximizing metrics for higher eigenvalues among all smooth metrics are known the exist, maximizing metrics for higher eigenvalues among $S^1$-invariant metrics on the annulus and M\"obius band do exist, except for the second normalized eigenvalue on the annulus. It is natural to then ask what happens for non-invariant metrics on the annulus and M\"obius band. For the first nonzero normalized Steklov eigenvalue, A. Fraser and R. Schoen \cite{FS3} proved that the supremum among all smooth metrics is the same as the supremum among $S^1$-invariant metrics on both the annulus and M\"obius band, and is achieved by the critical catenoid and the critical M\"obius band, respectively. On the other hand, for the higher eigenvalues, it is shown in \cite{FS4} that for $k \geq 2$ the supremum of the $k$-th normalized Steklov eigenvalue over all metrics on the M\"obius band (or respectively, annulus) is strictly bigger than the supremum 
over $S^1$-invariant metrics on the M\"obius band (or respectively, annulus) \cite[Theorem 5.2]{FS4}. Thus, for higher eigenvalues, we see that we get rather different results when we impose symmetry. When maximizing the second normalized eigenvalue among $S^1$-invariant metrics on the annulus, the conformal structure degenerates into two disks. In particular, the symmetry assumption does not prevent this type of degeneration from occurring. However other types of degenerations, as discussed in \cite{FS4}, cannot occur under the $S^1$ symmetry assumption. These results show that imposing symmetry prevents certain types of degenerations from occurring when maximizing higher eigenvalues.

In Section \ref{section:critical} we consider more generally the critical metrics of the Steklov eigenvalues among $S^1$-invariant metrics on the M\"obius band and show that the critical metrics are the induced metrics on explicit free boundary minimal M\"obius bands in $\mathbb{B}^4$ embedded by Steklov eigenfunctions. In particular, we obtain existence of a new explicit family of embedded free boundary minimal M\"obius bands in $\mathbb{B}^4$.

\begin{theorem} \label{theorem:critical}
The critical metrics of the normalized Steklov eigenvalues on the space of $S^1$-invariant metrics on the M\"obius band are (up to $\sigma$-homothety) the induced metrics on the embedded free boundary minimal M\"obius bands in $\mathbb{B}^4$ given explicitly by 
$u: [-t_{m,n}, t_{m,n}] \times S^1 / \sim \; \rightarrow \; \mathbb{B}^4$ with $u(t,\theta)=$
\[     
     \frac{1}{r_{m,n}}\left( m\sinh(nt) \cos(n\theta), m\sinh(nt) \sin(n\theta), 
         n\cosh(mt) \cos(m \theta), n\cosh(mt) \sin(m \theta) \right)
\]         
for each $m$ even and $n$ odd with $m > n$, where $t_{m,n}$ is the unique positive solution of $m \tanh(mt)=n\coth(nt)$ and $r_{m,n}=\sqrt{m^2\sinh^2(n t_{m,n})+n^2\cosh^2(m t_{m,n})}$.
\end{theorem}

Interestingly, these free boundary minimal M\"obius bands are the restrictions to $\mathbb{B}^4$, meeting the boundary of the ball orthogonally, of complete noncompact embeddings of the M\"obius band $\mathbb{R} \times S^1 / \sim$ into $\mathbb{R}^4$ of P. Mira \cite{M} and M.~E. de Oliveira \cite{D}. For {\em any} $m > n$ with $m, \, n \in \mathbb{N}$, these also give free boundary minimal immersions of the annulus into $\mathbb{B}^4$ that are critical for higher eigenvalues among rotationally symmetric metrics on the annulus (see Theorem \ref{theorem:critical-annulus} in section \ref{section:critical}).

In Section \ref{section:classification}
we prove that the induced metric on any $S^1$-invariant free boundary minimal annulus or M\"obius band in $\mathbb{B}^n$ is critical for some normalized Steklov eigenvalue on the space of $S^1$-invariant metrics on the annulus or M\"obius band. This allows us to explicitly classify all $S^1$-invariant free boundary minimal annuli and M\"obius bands in $\mathbb{B}^n$.
\begin{theorem} \label{theorem:classification}
The only $S^1$-invariant free boundary minimal annuli and M\"obius bands in $\mathbb{B}^n$ are those given explicitly in Theorem \ref{theorem:critical} and Theorem \ref{theorem:critical-annulus}.
\end{theorem}

{\em Acknowledgements}. The authors would like to thank R. Schoen and Martin Li for helpful discussions and for their interest in this work.

\section{Preliminaries and results for the disk and annulus} \label{section:annulus}

Let $M$ be a compact surface with boundary. First we recall the variational characterization of the Steklov eigenvalues:
\begin{equation*}
         \sigma_k 
         =\inf\left\{\frac{\int_M |\nabla u|^2dv_M}{\int_{\partial M}u^2dv_{\partial M}} \ : \ 
         \int_{\partial M}uu_j=0 \ \text{for} \ j=0,1,2,\ldots,k-1.\right\},
\end{equation*}
where $u_j$ is an eigenfunction corresponding to the eigenvalue $\sigma_j$, for $j=1,2,\ldots,k-1$.
If we scale a metric $g$ on $M$ by a factor $c>0$, then
\[
        \sigma_k(cg)=\frac{1}{\sqrt{c}} \sigma_k(g).
\]
Because of this scaling property, maximizing $\sigma_k(g)$ among metrics $g$ with fixed boundary length $L_g(\partial M)$ is equivalent to maximizing the (scale invariant) normalized eigenvalues
\[
       \bar{\sigma}_k(g):=\sigma_k(g) L_g(\partial M)
\]
over all metrics.
\begin{definition} \label{definition:equivalent}
We say that two surfaces $(M_1,g_1)$ and $(M_2,g_2)$ are {\it $\sigma$-homothetic} if there is a conformal diffeomorphism $\varphi: (M_1,g_1) \to (M_2,g_2)$ such that the pullback metric $\varphi^*(g_2)=\lambda^2 g_1$
with $\lambda$ constant on $\p M_1$. 
\end{definition}
It is clear from the variational charactization of the eigenvalues that if two surfaces are $\sigma$-homothetic then the normalized eigenvalues $\bar{\sigma}_k(M_j,g_j)$ coincide for $j=1,\,2$ and for all $k$. In particular, we can only hope to characterize surfaces up to $\sigma$-homothety by conditions on the Steklov spectrum.

As discussed in Section \ref{section:introduction}, the Euclidean round disk maximizes the first nonzero normalized Steklov eigenvalue $\bar{\sigma}_1$ among all smooth metrics on the disk \cite{W}. For the higher Steklov eigenvalues, for each $k \geq 1$ there is a sharp upper bound on $\bar{\sigma}_k(g)$ for any smooth metric $g$ on the disk (\cite{HPS}, \cite{GP}), but the upper bound is not achieved for any $k \geq 2$ (\cite{GP}, \cite{FS4}).

For surfaces with boundary, the next natural case to consider after the disk is the annulus. For the first normalized Steklov eigenvalue on the annulus, A. Fraser and R. Schoen \cite[Theorem 6.7]{FS3} proved existence of a maximizing metric and a sharp upper bound:
\begin{theorem}[\cite{FS3}] \label{theorem:annulus1}
For any metric $g$ on the annulus $M$ we have 
\[
         \sigma_1(g)L_g(\partial M)\leq (\sigma_1L)_{cc} \approx 4\pi/1.2
\]
with equality if and only if $(M,g)$ is $\sigma$-homothetic to the critical catenoid.
\end{theorem}
The critical catenoid is the unique portion of a suitably scaled catenoid centered at the origin inside the unit ball, that meets the boundary of the ball orthogonally. The explicit characterization of the maximizing metric in Theorem \ref{theorem:annulus1} follows from a minimal surface uniqueness theorem, characterizing the critical catenoid as the unique free boundary immersion of the annulus into $\mathbb{B}^n$ by first Steklov eigenfunctions. The main argument in the proof of this characterization involves showing that such a surface must be $S^1$-invariant. The $S^1$-invariant case is analyzed in detail in \cite[Section 3]{FS1}, where it is shown that the induced metric on the critical catenoid maximizes the first normalized Steklov eigenvalue among all rotationally symmetric metrics on the annulus.

X.-Q. Fan, L.~F. Tam and G. Yu \cite{FTY} extended the analysis of rotationally symmetric metrics on the annulus and determined the supremum of the $k$-th nonzero normalized Steklov eigenvalue over all rotationally symmetric metrics on the annulus and showed that the supremum is achieved for all $k>2$, but that the second nonzero normalized Steklov eigenvalue is not be achieved.

\begin{theorem}[\cite{FTY}]
Let $\sigma^{S^1}_k$ be the supremum of $k$-th normalized Steklov eigenvalue among all $S^1$-invariant metrics on the annulus. Then,

\vspace{1mm}

(i) $\sigma^{S^1}_2=4\pi$. Moreover, $\bar{\sigma}_2(g_T) \rightarrow 4\pi$ as $T \rightarrow \infty$, where $g_T=dt^2+d\theta^2$ on the cylinder $[0,T]\times \mathbb{S}^1$, and the supremum $4\pi$ is not achieved. 

\vspace{1mm}

(ii) $\sigma^{S^1}_{2k-1}=4k\pi/t_{1,0}$ for all $k \geq 1$, where $t_{1,0}$ is the unique positive solution of $\tanh t=1/ t$, and is achieved by the induced metric on the $k$-critical catenoid $u : [-t_{1,0}/k,t_{1,0}/k] \times S^1 \rightarrow \mathbb{B}^3$
\[
     u  (t,\theta) = \frac{1}{r_{1,0}} \left( \cosh (kt) \cos (k \theta), \cosh(kt) \sin(k \theta), kt \right)
\]
where $r_{1,0}=\sqrt{t_{1,0}^2+\cosh^2t_{1,0}}$. 

\vspace{1mm}

(iii) $\sigma_{2k}^{S^1}=4k\pi \tanh (kt_{k,1}/2)$ for $k>1$, where $t_{k,1}$ is the unique positive solution of 
$k\tanh(kt)=\coth(t)$, and is achieved by the induced metric from the free boundary minimal immersion given by
$u : [-t_{k,1},t_{k,1}] \times S^1 \rightarrow \mathbb{B}^4$
\[
         u(t,\theta)
      =\frac{1}{r_{k,1}}
         \left( k \sinh(t) \cos(\theta), k \sinh(t) \sin(\theta), \cosh(kt) \cos(k \theta), \cosh(kt) \sin(k \theta) \right)
\]
where $r_{k,1}=\sqrt{k^2\sinh^2(t_{k,1})+\cosh^2(kt_{k,1})}$.
\end{theorem}
Notice that in {\em(iii)}, $u(t,\theta)=u(-t, \theta+\pi)$ and so the image surface is a free boundary minimal M\"obius band in $\mathbb{B}^4$.
As we will see in the next section, these free boundary minimal M\"obius bands are maximizers for the Steklov eigenvalues on the space of $S^1$-invariant metrics on the M\"obius band.

\section{Maximizing metrics on the M\"obius band}
\label{section:mobius}

In \cite{FS3}, A. Fraser and R. Schoen proved existence of a metric that maximizes the first nonzero normalized Steklov eigenvalue among all smooth metrics on the M\"obius band. Moreover, they explicitly characterized the maximizing metric as the induced metric from a proper free boundary minimal embedding of the M\"obius band into $\mathbb{B}^4$  by first Steklov eigenfunctions, and obtained a sharp upper bound for the first eigenvalue (\cite[Theorem 7.5]{FS3}):

\begin{theorem}[\cite{FS3}]
For any metric $g$ on the M\"obius band $M$ we have 
\[
        \sigma_1(g)L_g(\partial M)\leq (\sigma_1L)_{cmb} =2\pi\sqrt{3}
\]
with equality with equality if and only if $(M,g)$ is $\sigma$-homothetic to the critical M\"obius band.
\end{theorem}

As in the case of the annulus, the explicit characterization of the maximizing metric follows from a minimal surface uniqueness theorem, characterizing the ``critical M\"obius band" (defined in Proposition \ref{prop:cmb}) as the unique free boundary immersion of the M\"obius band into $\mathbb{B}^n$ by first Steklov eigenfunctions. The main argument in the proof of this characterization involves showing that such a surface must be $S^1$-invariant. The $S^1$-invariant case is analyzed in detail in \cite[Proposition 7.1]{FS3}:

\begin{proposition}[\cite{FS3}] \label{prop:cmb}
There is a minimal embedding of the M\"obius band $\mathbb{R} \times S^1 / \sim$, with the identification $(t,\theta) \sim (-t,\theta+\pi)$, into $\mathbb R^4$ given by
\[ 
       u(t,\theta)
       =(2\sinh (t)\cos(\theta),2\sinh (t)\sin(\theta),\cosh (2t)\cos(2\theta),\cosh (2t)\sin(2\theta)).
\]
For a unique choice of $T$ the restriction of $\varphi$ to $[-T,T]\times S^1$ defines a proper embedding into a ball by first Steklov eigenfunctions. We may rescale the radius of the ball to $1$ to get the {\em critical M\"obius band}. Explicitly $T$ is the unique positive solution of $2\tanh 2t=\coth t$. Moreover, the maximum of 
$\sigma_1L$ over all $S^1$-invariant metrics on the M\"obius band is $2\pi \sqrt{3}$ and is uniquely achieved 
(up to conformal changes of the metric that are constant on the boundary) by the critical M\"obius band.
\end{proposition}

In this section, we extend the analysis of $S^1$-invariant metrics on the M\"obius band of \cite{FS3} and consider the problem of maximizing the $k$-th normalized Steklov eigenvalue on the M\"{o}bius band over all $S^1$-invariant metrics. We show that this problem is solvable for all $k$, i.e. for each $k$, among all $S^1$-invariant metrics on the M\"{o}bius band, there is a metric that maximizes the $k$-th normalized Steklov eigenvalue and it is achieved by a free boundary minimal M\"{o}bius band in $\mathbb{B}^4$. This is in contrast to the case of the annulus, where the supremum of the second normalized eigenvalue is not achieved \cite{FTY}.

Any Riemannian M\"obius band is conformal to $M_T:=[-T,T] \times S^1 / \sim$ for some $T >0$, where $(t,\theta)\sim(t',\theta')$ if $t'=-t$ and $\theta'=\theta+\pi$, with the flat metric $dt^2+d\theta^2$. We refer to $T$ as the {\em conformal modulus} of the M\"obius band. We consider $S^1$-invariant metrics on the M\"obius band; that is, conformal metrics on $M_T$, of the form
\[
                 g=f(t)^2(dt^2+d\theta^2)
\]
where $f:[-T,T]\rightarrow\mathbb{R}$ is a smooth function satisfying $f(t)=f(-t)$. Our goal is to determine the supremum, which we denote by $\sigma^{S^1}_k$, of the $k$-th nonzero normalized Steklov eigenvalue among all $S^1$-invariant metrics on the M\"obius band, 
\[
       \sigma^{S^1}_k:=\sup_{T>0} \; \sup \{ \bar{\sigma}_k(g) \, : \, g=f^2(t)(dt^2 + d\theta^2) \mbox{ on } M_T    
       \mbox{ with } f>0, \, f(-t)=f(t)  \}.
\] 

First, we fix the conformal class $T>0$ and an $S^1$-invariant metric $g=f(t)^2(dt^2+d\theta^2)$ on $M_T$.
The outward unit normal vector at a boundary point $(T,\theta)$ is given by $\eta=f(T)^{-1}\der{t}$. A function $u(t,\theta)$ on $M_T$ satisfying $u(t,\theta) = u(-t,\theta+\pi)$ is Steklov eigenfunction if it is a harmonic function on $M_T$ such that $u_\eta=\sigma u$ on $\partial M_T$ for some $\sigma \geq 0$. Notice that $u(t,\theta)$ is harmonic with respect to $g=f(t)^2(dt^2+d\theta^2)$ if it is harmonic with respect to the flat metric $dt^2+d\theta^2$, and thus satisfies the equation $u_{tt}+u_{\theta \theta}=0$. We may use the method of separation of variables to get $u(t,\theta) = \alpha(t)\beta(\theta)$, with $\alpha(t) = \alpha(-t)$ and $\beta(\theta)=\beta(\theta+\pi)$ and 
\begin{equation*}
\frac{\alpha''(t)}{\alpha(t)} = -\frac{\beta''(\theta)}{\beta(\theta)}= k^2.
\end{equation*}
We obtain solutions $u(t,\theta)$ for each nonnegative integer $k$ given by linear combinations of $\sinh(kt)\sin(k\theta)$ and $\sinh(kt)\cos(k\theta)$ when $n$ is odd, and $\cosh(kt)\sin(k\theta)$ and $\cosh(kt)\cos(k\theta)$ when $k$ is even. For $k=0$ the solutions are constants.

In order to be a Steklov eigenfunction we must have $u_\eta=\sigma u$ on the boundary, or $f(T)^{-1}u_t=\sigma u$ at the boundary point $(T,\theta)$. For $k=0$ we have $u(t,\theta)=a$, a constant, and $\sigma=0$. For $k \geq 1$ odd the eigenfunctions have $\alpha(t)=a \sinh (kt)$ and the condition is
\[
              k f(T)^{-1} \cosh(kT)=\sigma \sinh(kT).
\]
Therefore, $\sigma=k f(T)^{-1} \coth(kT)$. For $k \geq 1$ even the eigenfunctions have $\alpha(t)=a \cosh (kt)$ and the condition is
\[
              k f(T)^{-1} \sinh(kT)=\sigma  \cosh(kT).
\]
Therefore, $\sigma=k f(T)^{-1} \tanh(kT)$.
 
Thus, the nonzero Steklov eigenvalues are
\begin{equation*}
                  \lambda_k= \frac{2k}{f(T)}\tanh(2kT), \quad \text{and} \quad 
                  \mu_k= \frac{(2k-1)}{f(T)}\coth((2k-1)T),
\end{equation*}
$k=1,2,\ldots$. The length of the boundary is $L_g(\partial M_T)=2\pi f(T)$, and the normalized eigenvalues are
\begin{equation*}
                \bl_k= 4\pi k\tanh(2kT), \quad \text{and} \quad 
                \bm_k= 2\pi(2k-1)\coth((2k-1)T),
\end{equation*}
$k=1,2,\ldots$.
Notice that in each conformal class of metrics on the M\"obius band (i.e. for each $T >0$), the normalized eigenvalues $\bl_k$ and $\bm_k$ are constant on the space of $S^1$-invariant metrics in that conformal class (i.e. independent of $f$). We let $\bar{\sigma}_k(T)$ denote the $k$-th normalized Steklov eigenvalue of any $S^1$-invariant metric on $M_T$.
Then, the supremum of the $k$-th nonzero normalized Steklov eigenvalue among all $S^1$-invariant metrics on the M\"obius band is
\[
       \sigma^{S^1}_k=\sup_{T>0} \; \bar{\sigma}_k(T).
\]

In order to determine $\sigma^{S^1}_k$ and prove Theorem \ref{theorem:mobius} and Theorem \ref{theorem:critical} we will now prove a series of lemmas.

\begin{lemma}\label{LambdaMuMonotone}
Let $k,l \geq1$. Then
\begin{enumerate}
\item[(i)] $\bl_k<\bl_{k+1}$, $\bm_l<\bm_{l+1}$. Furthermore, $\bl_n<\bm_{n+1}$ for $n\geq 1$, and each $\bl_k$ and $\bm_l$ has multiplicity 2.
\item[(ii)] $\bl_k(T)$ is monotone increasing in $T$ and $\bm_l(T)$ is monotone decreasing in $T$.
\item[(iii)] ${\displaystyle \bl_k(\infty):=\lim_{T\rightarrow\infty}\bl_k(T) = 4\pi k}$ and ${\displaystyle \bm_l(\infty):=\lim_{T\rightarrow\infty}\bm_l(T) = 2\pi(2l-1)}$.
\end{enumerate} 
\end{lemma}
\begin{proof}
First, (i) and (iii) are clear by direct calculation. Now, (ii) follows from the fact that
\begin{equation*}
                \frac{d\bl_k}{dT} =8\pi k^2\sech^2(2kT)>0 \quad \text{and} \quad \frac{d\bm_l}{dT} 
                                           = -2\pi(2l-1)^2\csch^2((2l-1)T)<0.
\end{equation*}
\end{proof}

\begin{lemma}\label{TklExist}
There exists $T>0$ such that $\bl_k(T)=\bm_l(T)$ if and only if $l\leq k$. Moreover, $T$ is unique if it exists.
\end{lemma}
\begin{proof}
Let $F_{k,l}(T) = \bl_k(T) - \bm_l(T) = 2\pi\left(2k\tanh(2kT)-(2l-1)\coth\left((2l-1)T\right)\right)$. Then $F_{k,l}(T)$ is continuous on $(0,\infty)$ and
\begin{equation*}
                 \lim_{T\rightarrow0} F_{k,l}(T) 
                 = -\infty \quad \text{and} \quad \lim_{T\rightarrow\infty}F_{k,l}(T) 
                 = 2\pi(2k-(2l-1)).
\end{equation*}
Thus ${\displaystyle \lim_{T\rightarrow\infty}F_{k,l}(T)>0}$ if and only if $l\leq k$. Furthermore, $F_{k,l}(T)$ is monotone increasing on $(0,\infty)$ since $\bl_k(T)$ is monotone increasing and $\bm_l(T)$ is monotone decreasing. Hence there exists a unique $T>0$ for which $\bl_k(T)=\bm_l(T)$ if and only if $l\leq k$. 
\end{proof}

\begin{definition}
For $l\leq k$ let $T_{k,l}$ be the unique positive number such that
\begin{equation*}
         \bl_k(T_{k,l})=\bm_l(T_{k,l}).
\end{equation*}
\end{definition}

\begin{lemma}\label{TklLem}
For $l\leq k$, $T_{k,l}$ is decreasing in $k$ and increasing in $l$. 
\end{lemma}
\begin{proof}
Since $\bl_k(T)<\bl_{k+1}(T)$, we have that
\begin{equation*}
        \bm_l(T_{k,l}) =\bl_k(T_{k,l})<\bl_{k+1}(T_{k,l}).
\end{equation*}
Hence, $F_{k+1,l}(T_{k,l})>0$, where $F_{k,l}$ is as in the proof of Lemma \ref{TklExist}, and, again, 
\[
     \lim_{T\rightarrow 0} F_{k+1,l}(T)=-\infty.
\]
Hence $T_{k+1,l}<T_{k,l}$. Similarly, if $l+1\leq k$,
\begin{equation*}
         \bl_k(T_{k,l})=\bm_l(T_{k,l})<\bm_{l+1}(T_{k,l}),
\end{equation*}
and so $F_{k,l+1}(T_{k,l})<0$. Since ${\displaystyle \lim_{t\rightarrow\infty} F_{k,l+1}(T)>0}$, it follows that $T_{k,l}<T_{k,l+1}$.
\end{proof}

For fixed $k> 0$, let $s=\lfloor{\frac{k}{2}}\rfloor$. By Lemma \ref{TklLem}, we see that we can decompose $[0,\infty)$ as
\begin{equation*}
            [0,\infty) 
            = \bigcup_{j=0}^{s}\big( [T_{k-j,j}, T_{k-j,j+1})\cup[T_{k-j,j+1}, T_{k-j-1,j+1}) \big),
\end{equation*}
where we define $T_{k,0}=0$ and $T_{k,l}=\infty$ if $k < l$.

\begin{lemma}\label{SigmaBound}
For $k\geq1$, 
\begin{align*}
      \bs_{2k-1}&(T)=\bs_{2k}(T) \\
      &=\begin{cases}
      \bl_{k-j}(T) \leq \bl_{k-j}(T_{k-j,j+1}) & \mbox{if } T \in [T_{k-j,j}, T_{k-j,j+1}), \, 0 \leq j \leq s \\
      \bm_{j+1}(T) \leq \bl_{k-j}(T_{k-j,j+1}) & \mbox{if } T \in [T_{k-j,j+1}, T_{k-j-1,j+1}), \, 0 \leq j \leq s \\
      \bl_{k/2}(T) \leq \bl_{k/2}(\infty) & \mbox{if } T \in [T_{k-s,s}, \infty), \, s=\frac{k}{2}, \, k \mbox{ even} \\
      \bm_{(k+1)/2}(T) \leq \bl_{(k+1)/2}(T_{(k+1)/2, (k+1)/2})  
                    & \mbox{if } T\in[T_{k-s,s+1}, \infty), \, s=\frac{k-1}{2}, \, k \mbox{ odd} 
\end{cases}.
\end{align*}
\end{lemma}

\begin{proof}
We prove the result by induction on $j$. First consider the case when $j=0$ and suppose $T \in (0,T_{k,1})$. Then, since $\bl_{k}(T)$ is increasing in $T$ and $\bm_1(T)$ is decreasing in $T$ by Lemma \ref{LambdaMuMonotone}, we have that
\[
    \bl_1(T) < \bl_2(T) < \cdots < \bl_{k}(T) < \bl_{k}(T_{k,1}) 
    = \bm_1(T_{k,1}) <\bm_1(T) < \bm_2(T) < \cdots.
\]
Since each $\bl_n(T)$ has multiplicity two, $\bs_{2k-1}(T)=\bs_{2k}(T)=\bl_k(T) < \bl_k(T_{k,1})$. Now suppose $j=0$ and $T \in [T_{k,1},T_{k-1,1})$. Then,
\[
     \bl_1(T) < \cdots < \bl_{k-1}(T) < \bl_{k-1}(T_{k-1,1})
     =\bm_1(T_{k-1,1}) < \bm_1(T) \leq \bm_1(T_{k,1}) = \bl_{k}(T_{k,1}) < \bl_{k}(T).
\]
Since $\bm_1(T)$ and each $\bl_n(T)$ have multiplicity two, 
$\bs_{2k-1}(T)=\bs_{2k}(T)=\bm_1(T) \leq \bm_1(T_{k,1})=\bl_{k}(T_{k,1})$.

For $1 \leq j < s$, assuming the result holds for $j-1$, we show that it holds for $j$. First suppose $T \in [T_{k-j,j}, T_{k-j,j+1})$. 
By the induction hypothesis, $\bs_{2k-1}(T)=\bs_{2k}(T)=\bm_{j}(T)$ for $T \in [T_{k-j+1,j}, T_{k-j,j})$, and $\bs_{2k-1}(T_{k-j,j})=\bs_{2k}(T_{k-j,j})=\bm_{j}(T_{k-j,j})=\bl_{k-j}(T_{k-j,j})$. Since $\bm_j(T)$ is monotone decreasing in $T$, $\bl_{k-j}(T)$ is monotone increasing in $T$, and each has multiplicity two, we must have $\bs_{2k-1}(T)=\bs_{2k}(T)=\bl_{k-j}(T)$ for $T \in [T_{k-j,j}, T_{k-j,j}+\varepsilon)$ for some $\varepsilon >0$. Since for $T \in (T_{k-j,j}, T_{k-j,j+1})$,
\[       
       \bl_{k-j}(T) < \bl_{k-j}(T_{k-j,j+1})  =\bm_{j+1}(T_{k-j,j+1}) < \bm_{j+1}(T)
\]
it follows that $\bs_{2k-1}(T)=\bs_{2k}(T)=\bl_{k-j}(T) < \bl_{k-j}(T_{k-j,j+1})$ for all $T \in [T_{k-j,j}, T_{k-j,j+1})$. Now suppose $T \in [T_{k-j,j+1}, T_{k-j-1,j+1})$. Since 
\[
     \bs_{2k-1}(T_{k-j,j+1})=\bs_{2k}(T_{k-j,j+1})=\bl_{k-j}(T_{k-j,j+1})=\bm_{j+1}(T_{k-j,j+1}),
\]
$\bl_{k-j}(T)$ is monotone increasing in $T$, $\bm_{j+1}(T)$ is monotone decreasing in $T$, and each has multiplicity two, we must have $\bs_{2k-1}(T)=\bs_{2k}(T)=\bm_{j+1}(T)$ for $T \in [T_{k-j,j+1}, T_{k-j,j+1}+\varepsilon)$ for some $\varepsilon >0$. Since for $T \in (T_{k-j,j+1}, T_{k-j-1,j+1})$,
\[
      \bm_{j+1}(T) > \bm_{j+1}(T_{k-j-1,j+1}) = \bl_{k-j-1}(T_{k-j-1,j+1})  > \bl_{k-j-1}(T) 
\]
it follows that $\bs_{2k-1}(T)=\bs_{2k}(T)=\bm_{j+1}(T) \leq \bl_{k-j}(T_{k-j,j+1})$ for all $T \in [T_{k-j,j+1}, T_{k-j-1,j+1})$.

For $j=s$ and $k$ even, we have $s={k}/{2}$, and  
$[T_{k-j,j},T_{k-j,j+1})=[T_{k-s,s}, \infty)=[T_{{k}/{2},{k}/{2}},\infty)$. 
From above, we know that $\bs_{2k-1}(T)=\bs_{2k}(T)=\bm_{s}(T)$ for $T \in [T_{k-s+1,s}, T_{k-s,s})$. Since  
$\bs_{2k-1}(T_{k-s,s})=\bs_{2k}(T_{k-s,s})=\bm_{s}(T_{k-s,s})=\bl_{k-s}(T_{k-s,s})$, $\bm_s(T)$ is monotone decreasing in $T$, $\bl_{k-s}(T)$ is monotone increasing in $T$, and each has multiplicity two, we must have $\bs_{2k-1}(T)=\bs_{2k}(T)=\bl_{k-s}(T)$ for $T \in [T_{k-s,s}, T_{k-s,s}+\varepsilon)$ for some $\varepsilon >0$. Since for $T \in (T_{k-s,s},\infty)$,
\[
      \bl_{k-s}(T) < \bl_{k-s}(\infty)=4\pi(k-s) < 2\pi(2(s+1)-1)=\bm_{s+1}(\infty) < \bm_{s+1}(T)
\]
it follows that $\bs_{2k-1}(T)=\bs_{2k}(T)=\bl_{k-s}(T) = \bl_{k/2}(T) < \bl_{k/2}(\infty)$ for all $T \in [T_{k-s,s}, \infty)$.

For $j=s$ and $k$ odd, we have $s=(k-1)/2$. First suppose 
$T \in 
[T_{k-s,s}, T_{k-s,s+1})$. 
By the induction argument above with $j=s$ it follows that $\bs_{2k-1}(T)=\bs_{2k}(T)=\bl_{k-s}(T) < \bl_{k-s}(T_{k-s,s+1})$ for $T \in [T_{k-s,s}, T_{k-s,s+1})$.
Now suppose 
\[
       T \in 
       [T_{k-s,s+1}, T_{k-s-1,s+1}) = [T_{(k+1)/2,(k+1)/2}, \infty).
\] 
Since $\bs_{2k-1}(T_{k-j,j+1})=\bs_{2k}(T_{k-j,j+1})=\bl_{k-s}(T_{k-s,s+1})=\bm_{s+1}(T_{k-s,s+1})$, $\bl_{k-s}(T)$ is monotone increasing in $T$, $\bm_{s+1}(T)$ is monotone decreasing in $T$, and each has multiplicity two, we must have $\bs_{2k-1}(T)=\bs_{2k}(T)=\bm_{s+1}(T)$ for $T \in [T_{k-s,s+1}, T_{k-s,s+1}+\varepsilon)$ for some $\varepsilon >0$. Since for $T \in (T_{k-s,s+1}, T_{k-s-1,s+1})$,
\[
      \bm_{s+1}(T) > \bm_{s+1}(\infty) = 2\pi(2(s+1)-1) > 4\pi(k-s-1)=\bl_{k-s-1}(\infty)  > \bl_{k-s-1}(T) 
\]
it follows that $\bs_{2k-1}(T)=\bs_{2k}(T)=\bm_{s+1}(T) = \bm_{(k+1)/2}(T) \leq \bl_{(k+1)/2}(T_{(k+1)/2,(k+1)/2})$ for all $T \in [T_{(k+1)/2,(k+1)/2}, \infty)$.
\end{proof}

\begin{lemma}\label{FunctionLem}
Let 
\begin{equation*}
      f(t)=\sinh(t)\cosh(t)-t \quad \text{and} \quad  g(t) = \frac{\sinh(t)\cosh(t)-t}{t^2}.
\end{equation*}
Then $f(t)>0$ and $g'(t)>0$ for all $t>0$.
\end{lemma}

\begin{proof}
We have that
\begin{equation*}
\begin{split}
    f'(t) & = \cosh^2(t)+\sinh^2(t)-1\\
    & = 2\sinh^2(t)
\end{split}
\end{equation*}
which is positive for $t>0$. Since $f(0)=0$, it follows that $f(t)>0$ for $t>0$. 

Now 
\begin{equation*}
\begin{split}
     g'(t) = \frac{f'(t)t^2-2tf(t)}{t^4} &= \frac{2t(\sinh^2(t)+1)-2\sinh(t)\cosh(t)}{t^3}\\
     &= \frac{2t\cosh^2(t)-2\sinh(t)\cosh(t)}{t^3} \\
     & = \frac{2\cosh(t) \left( t\cosh(t)-\sinh(t) \right)}{t^3}. 
\end{split}
\end{equation*}
Let $h(t)= t\cosh(t)-\sinh(t)$. Then $h(0)=0$ and $h'(t)=t\sinh(t)>0$ for $t>0$, and so $f(t)>0$ for $t>0$. 
It follows that $g'(t)>0$ for $t>0$.
\end{proof}

\begin{lemma}\label{IntersectionHeight}
Let $x(a,b)$ be the unique positive solution of 
\begin{equation*}
     a\tanh(ax)=b\coth(bx)
\end{equation*}
for $a\geq b>0$. Let 
\begin{equation*}
     u(a,b) =a\tanh(ax(a,b)).
\end{equation*}
Then $u(a,b) <u(a+c,b-c)$ for $a\geq b>c>0$.
\end{lemma}

\begin{proof}
Differentiating the first equation with respect to $a$ yields
\begin{equation*}
      \tanh(ax)+a\sech^2(ax)\left(x+a\frac{\partial x}{\partial a}\right) 
      =-b^2\csch^2(bx)\cdot\frac{\partial x}{\partial a}
\end{equation*}
and so
\begin{equation*}
       \frac{\partial x}{\partial a} =\frac{-\tanh(ax)-ax\sech^2(ax)}{a^2\sech^2(ax)+b^2\csch^2(bx)}<0.
\end{equation*}
Similarly, 
\begin{equation*}
      \frac{\partial x}{\partial b} 
      = \frac{\coth(bx)-bx\csch^2(bx)}{a^2\sech^2(ax)+b^2\csch^2(bx)} 
      = \frac{\sinh(bx)\cosh(bx)-bx}{\sinh^2(bx)(a^2\sech^2(ax)+b^2\csch^2(bx))}>0,
\end{equation*} 
where we have used Lemma \ref{FunctionLem} to conclude its sign.

Now, since $u(a,b)=b\coth(bx(a,b))$ and $\frac{\partial x}{\partial a}<0$,
\begin{equation*}
      \frac{\partial u}{\partial a}=-b^2\csch^2(bx)\cdot\frac{\partial x}{\partial a}>0.
\end{equation*} 
Similarly,
\begin{equation*}
      \frac{\partial u}{\partial b} = a^2\sech^2(ax)\cdot\frac{\partial x}{\partial b}>0.
\end{equation*}
Hence,
\begin{equation*}
      \frac{\left(\frac{\partial u}{\partial a}\right)}{\left(\frac{\partial u}{\partial b}\right)} 
      = \frac{b^2(\sinh(ax)\cosh(ax)+ax)}{a^2(\sinh(bx)\cosh(bx)-bx)}
      >\frac{b^2(\sinh(ax)\cosh(ax)-ax)}{a^2(\sinh(bx)\cosh(bx)-bx)}\geq1
\end{equation*}
by Lemma \ref{FunctionLem} since $a\geq b$. Note that the inequality is strict when $a>b$. Thus, for $f(t) = u(a+t,b-t)$, 
\begin{equation*}
      f'(t)= \frac{\partial u}{\partial a}-\frac{\partial u}{\partial b},
\end{equation*}
and so $f'(t)>0$ for $t>0$. Hence $u(a,b)<u(a+c,b-c)$ for $a\geq b>c>0$.
\end{proof}

\begin{corollary}\label{FirstIntersectionMax}
For $k\geq l>c>0$ we have that
\begin{equation*}
      \bl_k(T_{k,l})<\bl_{k+c}(T_{k+c,l-c}).
\end{equation*}
\end{corollary}

\begin{proof}
By Lemma \ref{IntersectionHeight}, for $k\geq l>c>0$ we have that $u(2k,2l-1)<u(2k+2c,2l-1-2c)$. Hence $\bl_{k}(T_{k,l})<\bl_{k+c}(T_{k+c,l-c})$.
\end{proof}

In particular, this tells us that 
\begin{equation*}
     \bl_{k-j+1}(T_{k-j+1,j})<\bl_k(T_{k,1}),
\end{equation*}
for $2\leq j<s$ and, when $k$ is odd, 
\begin{equation*}
     \bl_{(k+1)/2}(T_{(k+1)/2,(k+1)/2})< \bl_k(T_{k,1}).
\end{equation*}
So, when $k$ is odd, by Lemma \ref{SigmaBound} we have that $\bs_{2k-1}(T)=\bs_{2k}(T)\leq \bl_{k}(T_{k,1})$, and 
when $k$ is even, $\bs_{2k-1}(T)=\bs_{2k}(T)\leq \max(\bl_{k/2}(\infty), \bl_k(T_{k,1}))$. 

\begin{lemma}\label{NoAsymptote}
For $k\geq 2$ even, 
\begin{equation*}
\bl_{k/2}(\infty)<\bl_{k}(T_{k,1})
\end{equation*}
\end{lemma}

\begin{proof}
Let $T_{k}$ be the unique positive number such that 
\[
      2k \tanh(2kT_{k}) = \frac{1}{T_{k}}.
\]
Then $\tanh(2kT_{k})=1/(2kT_{k})$ and so $T_{1/2}=2kT_{k}$. Also note that $T_{1/2}$ is the unique positive solution of $T=\coth T$ and $T_{1/2} \approx 1.2 < 2$. Therefore,
\begin{equation} \label{equation:bound}
     \bl_{\frac{k}{2}}(\infty)  =\frac{4\pi k}{2} 
     < \frac{4\pi k}{T_{1/2}} 
     = \frac{2\pi}{T_{k}} 
     = 4 \pi k \tanh(2k T_{k}).
\end{equation}
If $f(T)=\tanh T$, then $f'(T)=\sech^2 T$ and we have $f'(0)=1$ and $f'(T) < 1$ for $T>0$. Hence, $\tanh T < T$, or equivalently $1/T < \coth T$, for $T>0$. Arguing as in Lemma \ref{TklExist} and Lemma \ref{TklLem}, since $\lambda_k(T)=4\pi k \tanh(2kT)$ is increasing in $T$, $1/T$ and $\coth T$ are decreasing in $T$, and $1/T < \coth T$ for $T >0$, it follows that $T_{k} < T_{k,1}$. Using this  in (\ref{equation:bound}), since $\tanh$ is increasing, we have that
\[
     \bl_{\frac{k}{2}}(\infty) < 4 \pi k \tanh(2k T_{k}) < 4 \pi k \tanh(2k T_{k,1}) = \bl_k(T_{k,1}).
\]
\end{proof}

Using these results, we are now ready to prove Theorem \ref{theorem:mobius}. We first determine the supremum $\sigma^{S^1}_k$ of the $k$-th normalized Steklov eigenvalue among all $S^1$-invariant metrics on the M\"obius band, $\sigma^{S^1}_k=\sup_{T>0} \bs_k(T)$, and show that we can always find an $S^1$-invariant metric that achieves $\sigma^{S^1}_k$. We then use the eigenfunctions corresponding to these maximal Steklov eigenvalues to construct embedded free boundary minimal M\"{o}bius bands in $\mathbb{B}^4$.

\begin{proof}[Proof of Theorem \ref{theorem:mobius}]
It follows directly from Lemma \ref{SigmaBound}, Lemma \ref{FirstIntersectionMax} and Lemma \ref{NoAsymptote} that for each $k \geq 1$, 
\[
      \sigma^{S^1}_{2k-1}=\sigma^{S^1}_{2k}=4\pi \tanh(2kT_{k,1}),
\]
and is attained precisely when $T=T_{k,1}$.

Consider  $u: [-T_{k,1},T_{k,1}] \times S^1 / \sim \; \rightarrow \; \mathbb{R}^4$ given by
\[
     u(t,\theta)
     =\left( 2k \sinh(t) \cos(\theta), 2k \sinh(t) \sin(\theta), \cosh(2kt) \cos(2k \theta), \cosh(2kt) \sin(2k \theta) \right).
\]
First observe that 
\[
    |u(t,\theta)|^2=4k^2 \sinh^2(t)+\cosh^2(2kt).
\]
Suppose $u(t,\theta)=u(t',\theta')$. Then $|u(t,\theta)|^2=|u(t',\theta')|^2$, and if $t \ne 0$ the inverse function theorem implies that $t=t'$. If $t=0$ then we see immediately from the expression for $u$ that $t'=0$. In either case, $t=t'$, and the expression for $u$ implies that $\cos \theta=\cos \theta'$ and $\sin \theta=\sin \theta'$, so $\theta=\theta'$. Therefore, $u$ is an embedding.

Since the component functions of $u$ are Steklov eigenfunctions, they are harmonic functions. Furthermore, 
\begin{equation*}
\begin{split}
        \frac{\partial u}{\partial t} 
        & = \left(2k\cosh(t)\cos(\theta),2k\cosh(t)\sin(\theta), 2k\sinh(2kt)\cos(k\theta), k\sinh(2kt)\sin(2k\theta)\right),\\
       \frac{\partial u}{\partial \theta} 
       & =\left(-2k\sinh(t)\sin(\theta),2k\sinh(t)\cos(\theta),-2k\cosh(2kt)\sin(k\theta),2k\cosh(2kt)\cos(2k\theta)\right).
\end{split}
\end{equation*}
It follows that 
\[
      \frac{\partial u}{\partial t} \cdot\frac{\partial u}{\partial \theta}=0
\]
and 
\[
    \left|\frac{\partial u}{\partial t}\right|^2
    = 4k^2 \left( \cosh^2(t) + \sinh^2(2kt) \right)
    = 4k^2 \left( \sinh^2(t) + \cosh^2(2kt) \right)
    =\left| \frac{\partial u}{\partial \theta} \right|^2.
\]
Hence, $u$ is also conformal. Therefore $u$ is a minimal embedding. Moreover, $|u|$ is constant on $\partial M_{T_{k,1}}$, equal to 
\[
       R_k:=\sqrt{4k^2 \sinh^2(T_{k,1})+\cosh^2(2kT_{k,1})}. 
\]
It follows from the maximum principle that $u$ defines a surface contained in a ball centred at the origin of radius $R_{k}$, with boundary lying on the boundary of the ball. Furthermore, each component function of $u$ is a Steklov eigenfunction with eigenvalue $4\pi \tanh(2kT_{k,1})$, and so on $\partial M_{T_{k,1}}$ we have that 
\[
     \frac{\partial u}{\partial \eta} = 4\pi \tanh(2kT_{k,1}) \, u
\]
is orthogonal to the boundary of the ball. Therefore, $u(M_{T_{k,1}})$ is an embedded free boundary minimal surface in the ball of radius $R_k$. We may then rescale $u$ to obtain a free boundary minimal embedding of the M\"obius band into $\mathbb{B}^4$,
\[
   \tilde{u}:=\frac{1}{R_k} u: [-T_{k,1},T_{k,1}] \times S^1 / \sim \; \rightarrow \; \mathbb{B}^4.
\]
\end{proof}

\section{Critical metrics on the annulus and M\"obius band} \label{section:critical}

In this section we prove more generally that the critical metrics of the Steklov eigenvalues among $S^1$-invariant metrics on the M\"obius band are given by explicit embedded free boundary minimal M\"obius bands in $\mathbb{B}^4$, using the analysis of the Steklov spectrum of $S^1$-invariant metrics on the M\"obius band from the previous section. Similarly, we show that the critical metrics of the Steklov eigenvalues among $S^1$-invariant metrics on the annulus are given by explicit free boundary minimal annuli in $\mathbb{B}^3$ and $\mathbb{B}^4$, using the analysis of \cite{FTY}. As a consequence, we obtain new explicit families of free boundary minimal annuli and M\"obius bands in $\mathbb{B}^4$. Although it is relatively clear in our specific context, we begin this section by giving a precise definition of what we mean by `critical metric'. 

It follows from the explicit analysis of the Steklov eigenvalues of $S^1$-invariant metrics on the annulus and M\"obius band (\cite[Section 3]{FS1}, \cite{FTY}, \cite[Section 7]{FS3}, and Section \ref{section:mobius}) that the multiplicity of any Steklov eigenvalue of an $S^1$-invariant metric is either 1, 2, or 3 for the annulus, and either 2 or 4 for the M\"obius band. Moreover, the analysis shows that if $g(t)$ is a smooth family of $S^1$-invariant metrics on $M$ with $g(0)=g$, and $m$ is the dimension of the eigenspace $E_k(g)$, then there are smooth functions $\Lambda_1(t), \dots, \Lambda_m(t)$ such that $\Lambda_i(t)$, $i=1, \ldots, m$, is an eigenvalue of $(M,g(t))$, $\Lambda_1(0)= \cdots = \Lambda_m(0)=\sigma_k(g)$, and there exist $\delta>0$ and $1 \leq p, \, q \leq m$ such that 
\[
       \sigma_k(g(t))=\begin{cases} \Lambda_p(t) & \mbox{ for } t \in (-\delta, 0) \\
                                                      \Lambda_q(t) & \mbox{ for } t \in (0,\delta).
                               \end{cases}
\]
Furthermore, left and right hand derivatives of $\sigma_k(t)$ exist at $t=0$ and
\[
     \dot{\sigma}_k(0^-)=\dot{\Lambda}_p(0)
\]
and
\[
     \dot{\sigma}_k(0^+)=\dot{\Lambda}_q(0).
\]

\begin{definition} \label{definition:critical}
We say that an $S^1$-invariant metric $g$ on $M$ (the annulus or M\"obius band) is {\em critical} for the functional $\sigma_k$ on the space of $S^1$-invariant metrics on $M$ if for any smooth family of metrics $g(t)$ with $g(0)=0$ with fixed boundary length, $ \dot{\sigma}_k(0^-) \cdot \dot{\sigma}_k(0^+) \leq 0$.
\end{definition}

It is clear that if $g$ is a locally maximizing or locally minimizing metric of $\sigma_k$ on the space of $S^1$-invariant metrics, then $g$ is a critical metric. It follows from the explicit analysis of $S^1$-invariant metrics on the annulus and M\"obius band that these are the only types of critical metrics.

We first consider the case of the M\"obius band and prove Theorem \ref{theorem:critical}.

\begin{proof}[Proof of Theorem \ref{theorem:critical}]
It follows from the proof of Lemma \ref{SigmaBound} that for each $k \geq 1$ and $0 \leq j \leq s-1$, 
$\bs_{2k-1}(T)=\bs_{2k}(T)$ has a local maximum of $\bl_{k-j}(T_{k-j,j+1})$ at $T=T_{k-j,j+1}$
and a local minimum of $\bl_{k-j-1}(T_{k-j-1,j+1})$ at $T=T_{k-j-1,j+1}$, and for $k$ odd and $j=s$, $\bs_{2k-1}(T)=\bs_{2k}(T)$ has a local maximum of $\bl_{k-s}(T_{k-s,s+1})$ at $T=T_{k-s,s+1}$. Specifically, the critical points of the normalized Steklov eigenvalues $\bs_i(T)$ are $T=T_{k,l}$ for all $k \geq l$, and the critical values $\bs_i(T_{k,l})$ have multiplicity four with eigenspace spanned by
\begin{equation*}
\begin{cases}
    x(t,\theta)=\sinh((2l-1)t)\cos((2l-1)\theta)&\\
    y(t,\theta)=\sinh((2l-1)t)\sin((2l-1)\theta)&\\
    z(t,\theta)=\cosh(2kt)\cos(2k\theta)&\\
    w(t,\theta)=\cosh(2kt)\sin(2k\theta).&\\
\end{cases}
\end{equation*}
Consider  $u: [-T_{k,l},T_{k,l}] \times S^1 / \sim \; \rightarrow \; \mathbb{R}^4$ given by
\[
     u(t,\theta)
     =\frac{1}{R_{k,l}} \left( 2k \, x(t,\theta), 2k \, y(t,\theta), (2l-1) \, z(t,\theta), (2l-1) \, w(t,\theta) \right)
\]
where 
\[
       R_{k,l}:=\sqrt{4k^2 \sinh^2((2l-1)T_{k,l})+(2l-1)^2\cosh^2(2kT_{k,l})}. 
\]
As in the proof of Theorem \ref{theorem:mobius}, $u$ is a free boundary minimal embedding of the M\"obius band $M_{T_{k,l}}$ into $\mathbb{B}^4$.
\end{proof}

In a similar way, we are also able to characterize the critical metrics of the normalized Steklov eigenvalues among $S^1$-invariant metrics on the annulus, using the analysis of \cite{FTY}.

\begin{theorem} \label{theorem:critical-annulus}
The critical metrics of the normalized Steklov eigenvalues $\bar{\sigma}_k$ for $k \geq 1$ over all $S^1$-invariant metrics on the annulus are (up to $\sigma$-homothety) the induced metrics on: 

\vspace{2mm}

(i) 
The critical $n$-catenoid given by the immersion
$u: [-t_{1,0}/n,t_{1,0}/n] \times S^1 \rightarrow \mathbb{B}^3$
\[
    u(t,\theta) =  \frac{1}{r_n}(\cosh(nt)\cos(n\theta), \cosh(nt)\sin(n\theta), nt)
\]
for $n=1, \, 2, \ldots$, where $t_{1,0}$ is the unique positive solution of $ \tanh t=1/t$ and
\[
     r_{n}=\sqrt{\cosh^2(n t_{1,0}) +t_{1,0}^2}.
\]

\vspace{2mm}

(ii)
The free boundary minimal annuli given by the immersions
$u: [-t_{m,n},t_{m,n}] \times S^1  \; \rightarrow \; \mathbb{B}^4$ 
where  $u( t, \theta)=$
\[
     \frac{1}{r_{m,n}}\left( m\sinh(nt) \cos(n\theta), m\sinh(nt) \sin(n\theta), 
         n\cosh(mt) \cos(m \theta), n\cosh(mt) \sin(m \theta) \right)
\]
for any $m, \, n \in \mathbb{N}$ with $m > n$, where $t_{m,n}$ is the unique positive solution of $m \tanh(mt)=n\coth(nt)$ and
\[
     r_{m,n}=\sqrt{m^2\sinh^2(n t_{m,n})+n^2\cosh^2(m t_{m,n})}.
\]
\end{theorem}

\begin{proof}
The proof follows as in the proof of Theorem \ref{theorem:critical}, using \cite[Lemmas 2.4, 2.5, 2.6]{FTY}. 
In the case of the annulus, the normalized Steklov eigenvalues of an $S^1$-invariant metric $g$ on the annulus $M$ depend not only on the conformal modulus $T >0$ of the annulus, 
but also on a parameter 
\[
     \beta = \frac{4 \alpha}{(1+\alpha)^2}, 
\]
where $\alpha$ is the ratio of the lengths of the two boundary components of the annulus $(M,g)$. We may assume $\alpha \geq 1$, in which case we have $0 < \beta \leq 1$ with $\beta = 1$ if and only if $\alpha=1$.
It follows from \cite[Lemmas 2.4, 2.5]{FTY} that the critical points (local maxima and minima)
of the normalized Steklov eigenvalues $\bs_i(\beta,T)$ among metrics with $\beta$ fixed are $T=t_{m,n}(\beta)$ (see \cite[Definition 2.1]{FTY}) for all $m$, $n$ with $m/n > \alpha$. However, by \cite[Proof of Lemma 2.6(i)]{FTY}, 
\[
      \frac{d}{d\beta} \bs_i(\beta,t_{m,n}(\beta)) >0.
\]
Therefore the critical points of the normalized Steklov eigenvalues among all $S^1$-invariant metrics are the points $(\beta,T)$ with $\beta=1$, $T=t_{m,n}(1)$ with $m > n$. For $n=0$, the corresponding critical value has multiplicity three with eigenspace spanned by
\begin{equation*}
\begin{cases}
    x(t,\theta)=\cosh(mt) \cos(m\theta)&\\
    y(t,\theta)=\cosh(mt)\sin(m\theta)&\\
    z(t,\theta)=t.&
\end{cases}
\end{equation*}
For $n \geq 1$, the corresponding critical value has multiplicity four with eigenspace spanned by
\begin{equation*}
\begin{cases}
    x(t,\theta)=\sinh(nt)\cos(n\theta)&\\
    y(t,\theta)=\sinh(nt)\sin(n\theta)&\\
    z(t,\theta)=\cosh(mt)\cos(m\theta)&\\
    w(t,\theta)=\cosh(mt)\sin(m\theta).&\\
\end{cases}
\end{equation*}
As in the proof of Theorem \ref{theorem:mobius}, we see that the critical metrics of the normalized Steklov eigenvalues are achieved by the induced metrics from the free boundary minimal immersions given in {\em(i)} and {\em (ii)}, respectively, in the statement of the theorem.
\end{proof}

\section{Classification of $S^1$-invariant free boundary annuli and M\"obius bands} \label{section:classification}

In this section we give an explicit classification of all $S^1$-invariant free boundary minimal annuli and M\"obius bands in $\mathbb{B}^n$. Specifically, we show that the free boundary minimal annuli in Theorem \ref{theorem:critical-annulus} and the free boundary minimal M\"obius bands of Theorem \ref{theorem:critical} are the only $S^1$-invariant free boundary minimal annuli and M\"obius bands in $\mathbb{B}^n$. Here, we say that a free boundary minimal annulus or M\"obius band is {\em $S^1$-invariant} if the induced metric is $S^1$-invariant (in the sense discussed in section \ref{section:mobius} and \cite{FTY}). In order to give the classification, we show that the induced metric on any $S^1$-invariant free boundary annulus or M\"obius band in $\mathbb{B}^n$ is critical for some normalized Steklov eigenvalue on the space of $S^1$-invariant metrics. The result then follows from Theorem \ref{theorem:critical} and Theorem \ref{theorem:critical-annulus}. 

Let $M$ be a compact surface with boundary, and let $g(t)$ be a smooth family of metrics on $M$ with $\dot{g}(t)=h(t)$, where $h(t)\in S^2(M)$ is a smooth family of symmetric $(0,2)$-tensor fields on $M$. Denote by $E_k(g(t))$ the eigenspace corresponding to the $k$-th Steklov  eigenvalue $\sigma_k(t)$ of $(M,g(t))$. In \cite[Lemma 2.5]{FS2} it is shown that $\sigma_k(t)$ is a Lipschitz function of $t$, and if $\dot{\sigma}_k(t)$ exists, then
\[
       \dot{\sigma}_k(t) = Q_h(u)
\] 
for any $u \in E_k(g(t))$ with $||u||_{L^2}=1$, where 
\[
      Q_h(u)=-\int_M \la \tau(u) , h \ra \; da_t  -\frac{\sigma_k(t)}{2} \int_{\p M} u^2 h(T,T) \; ds_t.
\]
Here $T$ is the unit tangent to $\p M$ for the metric $g(t)$ and $\tau(u)$ denotes the stress-energy tensor of $u$ with respect to the metric $g(t)$,
\[
    \tau(u)= du \otimes du -\frac{1}{2} |\n u|^2 g.
\]
For a family of metrics $g(t)$ with fixed boundary length, the length constraint on the boundary translates to 
\[
      \int_{\p M}h(T,T)\ ds_t=0.
\]      
We let $S^2_0(M,g)$ denote the space of smooth symmetric $(0,2)$-tensor fields on $M$ satisfying 
$\int_{\partial M} h(T,T) \ ds=0$, 
where $T$ is the unit tangent vector to $\partial M$ with respect to the metric $g$.

\begin{lemma} \label{lemma:continuous}
Let $g(t)$ be a smooth family of metrics on a compact surface $M$ with boundary, with $g(0)=g$.
Suppose $\sigma$ is a Steklov eigenvalue of $(M,g)$ with multiplicity $m$. Let $k_1< \cdots < k_m$ be such that 
$\sigma_{k_1}(g) = \cdots = \sigma_{k_m}(g)=\sigma$, and let
$\mathcal{E}(t)$ be the direct sum of the eigenspaces corresponding to the distinct eigenvalues among $\sigma_{k_1}(g(t)), \ldots, \sigma_{k_m}(g(t))$. Then the $L^2$-orthogonal projection
\[
      P_t: L^2(\partial M, g(t)) \rightarrow \mathcal{E}(t)
\]
is continuous at $t=0$.
\end{lemma}

\begin{proof}
Since $\sigma_k(g(t))$ is a Lipschitz function of $t$ for every $k \geq 1$, first we note that for $|t|<\epsilon$, 
$\sigma_{k_1}(g(t)) > \sigma_{k_1 -1}(g(t))$ and $\sigma_{k_m}(g(t)) < \sigma_{k_m +1}(g(t))$. Therefore, $\mathcal{E}(t)$ has dimension $m$ for $|t|<\epsilon$. Let $u_1(t), \ldots, u_m(t)$ be an orthonormal basis of $\mathcal{E}(t)$ with respect to $L^2(\partial M, g(t))$, such that $u_j(t)$ is a Steklov eigenfunction with eigenvalue $\sigma_{k_j}(g(t))$.  Elliptic boundary estimates give bounds on $u_j(t)$ up to $\partial M$ that are independent of $t$. Hence for any sequence $t_i \rightarrow 0$ there is a subsequence, which we continue to denote by $t_i$, such that $u_j(t_i)$ converges in $C^2(M)$ to a Steklov eigenfunction $u_j$ of $(M,g)$ with eigenvalue $\sigma_{k_j}(g)=\sigma$ (see \cite[Proof of Theorem 1.1]{FS4} for details). It follows that $u_1, \ldots , u_m$ is an orthonormal basis of $\mathcal{E}(0)$, and if $f(t) \in L^2(\partial M, g(t))$ is a family of functions varying smoothly in $t$  then
\[
       \lim_{i \rightarrow \infty} P_{t_i}(f(t))
       = \sum_{k=1}^m \left( \int_{\partial M} f(t_i) \, u_k(t_i) \right) u_k(t_i)
        \stackrel{i \rightarrow \infty}{\longrightarrow} 
        \sum_{k=1}^m \left( \int_{\partial M} f(0) \, u_k \right) u_k = P_0(f(0)).
\]
\end{proof}

Recall (see \cite[Lemma 2.2]{FS1}) that if $M$ is a compact surface with nonempty boundary and $u=(u_1, \ldots , u_n): M \rightarrow \mathbb{B}^n$ is a proper branched immersion, then $u(M)$ is a free boundary minimal surface in $\mathbb{B}^n$ (i.e. $u(M)$ is a minimal surface in $\mathbb{B}^n$ that meets $\partial \mathbb{B}^n$ orthogonally) if and only if $u_1, \ldots, u_n$ are Steklov eigenfunctions of $(M, g)$ with eigenvalue 1, where $g$ is the induced metric; i.e. $g=u^*(\delta)$ where $\delta$ is the Euclidean metric on $\mathbb{B}^n$.

In \cite[Proof of Proposition 2.4]{FS2} it was shown that if $(M,g)$ is a Riemannian surface with boundary such that the quadratic form $Q_h$ is indefinite on $E_k(g)$, then there there are $k$-th eigenfunctions $u_1, \ldots, u_n$, $n \geq 2$, such that $u=(u_1, \ldots, u_n): M \rightarrow \mathbb{B}^n$ is a proper conformal branched immersion, and hence $u(M)$ is a free boundary minimal surface in $\mathbb{B}^n$. Here we prove a converse of this. This is an analog of \cite[Lemma 3.1]{EI}.

\begin{lemma} \label{lemma:indefinite}
Let $M$ be a compact surface with nonempty boundary, and suppose $u=(u_1, \ldots, u_k): M \rightarrow \mathbb{B}^n$ is a (possibly branched) free boundary minimal immersion. Then, for all $h \in S_0^2(M, g)$, the quadratic form $Q_h$ is indefinite on $E_k(g)$, where $g=u^*(\delta)$ and $k$ is such that $\sigma_k(g)=1$ (i.e. $u_1, \ldots, u_n$ are Steklov eigenfunctions with eigenvalue $\sigma_k$).
\end{lemma}

\begin{proof}
Since $g=u^*(\delta)=\sum_{i=1}^n du_i \otimes du_i$, $|\nabla u|^2=2$ and
\[
     \sum_{i=1}^n du_i \otimes du_i -\frac{1}{2} |\nabla u|^2 g=0.
\]
Since $u(\partial M) \subset \partial \mathbb{B}^n$, we have $\sum_{i=1}^n u_i^2=1$. Therefore,
\[
     \sum_{i=1}^n Q_h(u_i)=-\frac{\sigma_k(g)}{2} \int_{\partial M} h(T,T) \; ds =0
\]
if $h \in S_0^2(M, g)$. It follows that $Q_h$ is indefinite on $E_k(g)$ for all $h \in S_0^2(M, g)$.
\end{proof}

Using Lemma \ref{lemma:indefinite} we now show that the induced metric on any $S^1$-invariant free boundary minimal annulus or M\"obius band is critical for some Steklov eigenvalue.

\begin{proposition} \label{proposition:critical}
Let $M$ be the annulus or M\"obius band, and suppose $u=(u_1, \ldots, u_k): M \rightarrow \mathbb{B}^n$ is a free boundary branched minimal immersion such that $g:=u^*(\delta)$ is $S^1$-invariant. Then $g$ is critical for the $\sigma_k$ functional on the space of $S^1$-invariant metrics on $M$, for some $k \geq 1$.
\end{proposition}

\begin{proof}
Since $u=(u_1, \ldots, u_n): M \rightarrow \mathbb{B}^n$ is a free boundary minimal immersion,  $u_1, \ldots, u_n$ are Steklov eigenfunctions of $(M,g)$ with eigenvalue 1. Suppose the eigenspace corresponding to the eigenvalue 1 has mulitplicity $m$, and let $k_1< \cdots < k_m$ be such that $\sigma_{k_1}(g)= \cdots = \sigma_{k_m}(g)=1$.

Let $g(t)$ be a smooth family of $S^1$-invariant metrics on $M$ with fixed boundary length with $g(t)=g$, and let $h=\dot{g} \in S^2_0(M,g)$. Then there are smooth functions $\Lambda_1(t), \dots, \Lambda_m(t)$ such that each $\Lambda_i(t)$, $i=1, \ldots, m$, is an eigenvalue of $(M,g(t))$ and $\Lambda_1(0)= \cdots = \Lambda_m(0)=1$. Furthermore:
\begin{align*}
      \dot{\sigma_{k_1}}(0^-) &= \max \{ \dot{\Lambda}_1(0), \ldots \dot{\Lambda}_m(0) \}  \\
      \dot{\sigma_{k_1}}(0^+) &= \min \{ \dot{\Lambda}_1(0), \ldots \dot{\Lambda}_m(0) \}
\end{align*}
and
\begin{align*}
      \dot{\sigma}_{k_m}(0^+) &= \max \{ \dot{\Lambda}_1(0), \ldots \dot{\Lambda}_m(0) \} \\
      \dot{\sigma}_{k_m}(0^-) &= \min \{ \dot{\Lambda}_1(0), \ldots \dot{\Lambda}_m(0) \}.
\end{align*}

\vspace{1mm}
As in the proof of Lemma \ref{lemma:continuous}, let $\mathcal{E}(t)$ be the direct sum of eigenspaces corresponding to the distinct eigenvalues among $\Lambda_1(t), \ldots, \Lambda_m(t)$. It follows from Lemma \ref{lemma:indefinite} that $Q_h$ is indefinite on $\mathcal{E}(0)$. Therefore, there exists $u \in \mathcal{E}(0)$ such that $Q_h(u)=0$. Let $u(t) \in \mathcal{E}(t)$ be a family of functions varying smoothly in $t$ such that $\lim_{t \rightarrow 0} u(t)=u$. Then $u(t)=a_1 (t) u_1(t) + \cdots + a_m(t) u_s(t)$ where $u_1(t) , \ldots, u_s(t)$ are eigenfunctions corresponding to distinct eigenvalues among $\Lambda_1(t), \ldots, \Lambda_m(t)$ with $\| u_j(t) \|_{L^2(\partial M, g(t))}=1$ and $a_j(t) \in \mathbb{R}$ for $j=1, \ldots, s$.
It follows from Lemma \ref{lemma:continuous} that $Q_h$ is continuous on $\mathcal{E}(t)$ at $t=0$, and so
\begin{align*}
     0 & = Q_h(u) = \lim_{t \rightarrow 0^+} Q_h(u(t)) \\
        & = \lim_{t \rightarrow 0^+} \left( a_1(t)^2 Q_h(u_1(t)) + \cdots + a_s(t)^2Q_h(u_s(t)) \right) \\
        & = \lim_{t \rightarrow 0^+} ( a_1(t)^2 \dot{\sigma}_{k_{i_1}}(t) + \cdots + a_s(t)^2 \dot{\sigma}_{k_{i_s}}(t) ) \\
        & =  a_1(0)^2 \dot{\sigma}_{k_{i_1}}(0^+) + \cdots + a_s(0)^2 \dot{\sigma}_{k_{i_s}}(0^+) 
\end{align*}
for some $1 \leq i_1 <  \cdots < i_s \leq m$, and we must have $\dot{\sigma}_{k_1}(0^+) \leq 0$ and $\dot{\sigma}_{k_m}(0^+) \geq 0$. Similarly, we must have $\dot{\sigma}_{k_1}(0^-) \geq 0$ and $\dot{\sigma}_{k_m}(0^-)\leq 0$. Therefore, $g$ is a critical metric of both $\sigma_{k_1}$ and $\sigma_{k_m}$ on the space of $S^1$-invariant metrics on $M$.
\end{proof}

Combining Proposition \ref{proposition:critical} and Theorems \ref{theorem:critical} and \ref{theorem:critical-annulus}, we obtain a classification of all $S^1$-invariant free boundary annuli and M\"obius bands in $\mathbb{B}^n$.

\begin{customthm}{\ref{theorem:classification}}
The only $S^1$-invariant free boundary minimal annuli and M\"obius bands in $\mathbb{B}^n$ are those given explicitly in Theorem \ref{theorem:critical} and Theorem \ref{theorem:critical-annulus}.
\end{customthm}

\bibliographystyle{plain}

\end{document}